\documentclass[11pt]{amsart}

\usepackage[T1]{fontenc}
\usepackage[utf8]{inputenc}
\usepackage{lmodern}
\usepackage{amsmath,amssymb,amsthm}
\usepackage{microtype}
\usepackage{hyperref}

\hypersetup{
  colorlinks=true,
  linkcolor=blue,
  citecolor=blue,
  urlcolor=blue
}

\newtheorem{theorem}{Theorem}
\newtheorem{lemma}[theorem]{Lemma}

\newtheorem*{theorem*}{Theorem}
\newtheorem{corollary}[theorem]{Corollary}
\theoremstyle{definition}
\newtheorem{definition}[theorem]{Definition}
\newtheorem{remark}[theorem]{Remark}

\newcommand{\R}{\mathbb{R}}
\newcommand{\C}{\mathbb{C}}
\newcommand{\Q}{\mathbb{Q}}

\title{On exotic Diophantine triples in $\R[X]$}
\author{Ana Jurasi\'c}
\address[A. Jurasi\'c]{Faculty of Mathematics\\University of Rijeka\\ Radmile Matej\v{c}i\'c 2\\51000 Rijeka\\Croatia}
\email{ajurasic@math.uniri.hr}
\date{}

\begin{document}
\maketitle

\begin{abstract}
Originally, an exotic Diophantine triple is a set $\{a,b,c\}$ of distinct nonzero rational numbers for which
\[
 a+1,\quad b+1,\quad c+1,\quad ab+1,\quad ac+1,\quad bc+1,\quad abc+1
\]
are all perfect squares.  We prove that there is no such triple in $\R[X]$, with at least one nonconstant element, if none of $a,b,c$ is equal to $1$.  Equivalently, under the distinct nonzero convention, every exotic Diophantine triple in $\R[X]$ with a nonconstant element must contain the element $1$.   
\end{abstract}



\section{Introduction}
A classical Diophantine $m$-tuple is a set of $m$ distinct positive integers such
that the product of any two of its distinct elements, increased by $1$, is a square. This problem has a long history, beginning with examples of Diophantus in the rational case and Fermat in the integer case (see \cite{Dujella2024}). Since then, many variants have been considered, obtained either by changing the observed ring or by modifying the square conditions imposed on the elements. \\

Recently, Dujella, Kazalicki, and Petri\v{c}evi\'c \cite{DKP} introduced the following seven-square variant of the Diophantine tuple condition over the rational numbers. Here is a more general definition:

\begin{definition}\label{def:exotic}
A set $\{a,b,c\}$ of distinct nonzero elements of a commutative ring $R$ with $1$ is called an \emph{exotic Diophantine triple} if
\begin{equation}\label{eq:seven}
 a+1,\quad b+1,\quad c+1,\quad ab+1,\quad ac+1,\quad bc+1,\quad abc+1
\end{equation}
are all perfect squares of some elements of $R$.
\end{definition}

Equivalently, $\{1,a,b,c\}$ is a Diophantine quadruple and $abc+1$ is also a square. Obviously, for every exotic triple, sets $\{a,b,c\}$ and $\{1,ab,c\}$ are Diophantine triples. Dujella, Kazalicki, and Petri\v{c}evi\'c \cite{DKP} proved that infinitely many such triples exist over $\Q$, while no such triples exist in positive integers. The closely related four-square condition that $ab+1$, $ac+1$, $bc+1$, and $abc+1$ are all perfect squares was considered by Dujella and Szalay \cite{DS}, who proved that there are infinitely many such triples $\{a,b,c\}$ of positive integers. Recently, Jurasi\'{c} \cite{AJ} proved that in $\R[X]$ there exist infinitely many four-square polynomial triples $\{a,b,c\}$.

The nonexistence of exotic Diophantine triples in the polynomial ring
$\mathbb Z[X]$ is quite obvious, since any polynomial triple over $\mathbb Z[X]$ would yield an exotic
Diophantine triple in positive integers. Motivated by the recent rational seven-square problem of Dujella, Kazalicki, and Petri\v{c}evi\'c \cite{DKP}, we consider the corresponding analogue in the polynomial ring $\R[X]$.   

There are immediate nonconstant examples if the element $1$ is allowed in such a triple.  For instance, for every nonconstant $k\in\R[X]$, the set
\begin{equation}\label{eq:family-with-one}
 \{k^2-1,\ (k+1)^2-1,\ 1\}
\end{equation}
is an exotic triple in $\R[X]$, since the seven expressions in \eqref{eq:seven} are
\[
 k^2,\quad (k+1)^2,\quad (\sqrt2)^2,\quad (k^2+k-1)^2,\quad k^2,\quad (k+1)^2,\quad (k^2+k-1)^2.
\]

The question considered here is whether, under the usual convention that the entries are distinct and nonzero, an exotic triple over $\R[X]$ can be nonconstant when the value $1$ is excluded.  The answer is negative. In other words, in $\R[X]$ the presence of $1$ in an exotic Diophantine triple is forced.  The family \eqref{eq:family-with-one} shows that this conclusion is sharp. The proof uses the regularity theorem of Filipin and Jurasi\'c \cite{FJ} for polynomial Diophantine quadruples in $\R[X]$. 

\begin{theorem}\label{thm:main}
There is no set $\{a,b,c\}\subset \R[X]$ of distinct nonzero polynomials such that:
\begin{enumerate}
\item at least one of $a,b,c$ is nonconstant,
\item none of $a,b,c$ is equal to $1$,
\item all seven polynomials in \eqref{eq:seven} are perfect squares in $\R[X]$.
\end{enumerate}
\end{theorem}

\section{Preliminaries}

We first record that, once $0$ and $1$ are excluded, a nonconstant exotic triple over $\R[X]$ has no constant element.

\begin{lemma}\label{lem:no-constant}
Let the set $\{a,b,c\}\subset\R[X]$, with at least one nonconstant element, satisfy the seven-square condition \eqref{eq:seven}.  If the elements are nonzero and none of them is equal to $1$, then all three elements are nonconstant.
\end{lemma}

\begin{proof}
Assume that $k\in\{a,b,c\}$ is constant.  By hypothesis, $k\ne0,1$.  Let $f$ be a nonconstant element of the triple.  The seven-square condition gives polynomials $x,y\in\R[X]$ such that
\[
 f+1=x^2,\qquad kf+1=y^2.
\]
Eliminating $f$ gives
\begin{equation}\label{eq:constant-elim}
 y^2-kx^2=1-k.
\end{equation}
Then \eqref{eq:constant-elim} factors in $\C[X]$ as
\[
 (y-x\sqrt{k})(y+x\sqrt{k})=1-k.
\]
The right-hand side is a nonzero constant, so both factors are constants.  Hence $x$ is constant, so $f=x^2-1$ is constant, a contradiction.  
\end{proof}

In the proof of the main theorem, we use the regularity theorem for polynomial Diophantine quadruples over $\R[X]$.

\begin{theorem}[Filipin and Jurasi\'c \cite{FJ}]\label{thm:FJ}
Let $\{a,b,c,d\}$ be a set of four distinct nonzero polynomials in $\R[X]$, not all constant.  Suppose that the product of any two distinct elements of that set, increased by $1$, is a perfect square in $\R[X]$.  Then the quadruple is regular; equivalently, for any labeling $a,b,c,d$ of the four polynomials in that set,
\begin{equation}\label{eq:regularity}
 (a+b-c-d)^2=4(ab+1)(cd+1).
\end{equation}
\end{theorem}

Using the same notation as in \cite{DKP}, the expression \eqref{eq:regularity} can be replaced by $r_4(a,b,c,d)=0$, where
\[
 r_4(a,b,c,d)=(a+b-c-d)^2-4(ab+1)(cd+1).
\]
Then, $r_4$ is a symmetric polynomial in $a,b,c,d$. \\

Also, in the proof of the main theorem, we will use the following result:
\begin{lemma}\label{lem:equal-degree}
Let $p,q\in\R[X]$ be nonconstant polynomials with $\deg p=\deg q$.  If
\begin{equation}\label{eq:PQ-square}
 (p^2-1)(q^2-1)+1
\end{equation}
is a square in $\R[X]$, then either
\[
 q=\pm p\pm1,
\]
or
\[
 q=(\pm\sqrt2\pm1)p,
\]
where the signs $\pm$ are independent.
\end{lemma}

\begin{proof}
Let $h\in\R[X]$ satisfy
\begin{equation}\label{h^2}
 h^2=(p^2-1)(q^2-1)+1=p^2q^2-p^2-q^2+2.
\end{equation}
Replacing $h$ by $-h$ if necessary, assume that $h$ has the same leading coefficient as $pq$.  Since $p$ and $q$ have the same positive degree, it holds $p^2+q^2-2\neq 0$,
\[
 \deg(pq+h)=2\deg p,
\]
and
\[
 \deg(p^2+q^2-2)=2\deg p.
\]
 By \eqref{h^2}, it holds $(pq-h)(pq+h)=p^2+q^2-2$, so $pq-h$ is a constant. Write $pq-h=\lambda$, with $\lambda\in\R$.  Substituting $h=pq-\lambda$ into \eqref{h^2}, gives
\begin{equation}\label{eq:lambda-eq}
 (q-\lambda p)^2=(\lambda^2-1)p^2+2-\lambda^2.
\end{equation}

If $\lambda^2\notin\{1,2\}$, then the right-hand side of \eqref{eq:lambda-eq} has the form $\alpha p^2+\beta$ with $\alpha,\beta\in\R$ and $\alpha\beta\ne0$.  Over $\C[X]$, equation \eqref{eq:lambda-eq} gives a factorization
\[
 (q-\lambda p-p\sqrt{\alpha})(q-\lambda p+p\sqrt{\alpha})=\beta.
\] Both factors are constants, since their product is a nonzero constant.  Hence, their difference $-2p\sqrt{\alpha}$ is constant, contradicting the assumption that $p$ is nonconstant. 

If $\lambda^2=1$, then \eqref{eq:lambda-eq} gives $(q-\lambda p)^2=1$, so $q=\lambda p\pm1=\pm p\pm1$.  If $\lambda^2=2$, then \eqref{eq:lambda-eq} gives $(q-\lambda p)^2=p^2$, so $q=(\lambda\pm1)p=(\pm\sqrt2\pm1)p$.
\end{proof}

\section{Proof of the main theorem}

\begin{proof}[Proof of Theorem \ref{thm:main}]Assume that $\{a,b,c\}\subset\R[X]$ satisfies the hypotheses.  By Lemma \ref{lem:no-constant}, $a$, $b$, $c$ are nonconstant. Let
\[
 \deg a\le \deg b\le \deg c.
\] Since $a+1,b+1,c+1$ are perfect squares, write
\begin{equation}\label{eq:xyz}
 a=x^2-1,\qquad b=y^2-1,\qquad c=z^2-1,
\end{equation}
with $x,y,z\in\R[X]$.  Then $x,y,z$ are nonconstant and
\[
 \deg x\le \deg y\le \deg z.
\]
Let $r\in\R[X]$ satisfy
\begin{equation}\label{eq:r}
 ab+1=r^2.
\end{equation}
Changing $r$ to $-r$ if necessary, assume that $r$ has the same leading coefficient as $xy$.  From \eqref{eq:xyz} and \eqref{eq:r}, we have
\begin{equation}\label{eq:r-square}
 r^2=(x^2-1)(y^2-1)+1=x^2y^2-x^2-y^2+2
\end{equation}
and $\deg r=\deg x+
\deg y$.

The set $\{1,a,b,c\}$ is a Diophantine quadruple in $\R[X]$, with at least one nonconstant element.  By Theorem \ref{thm:FJ}, this quadruple is regular.  Thus $r_4(1,a,b,c)=0$.  Since $r_4$ is symmetric, we may write the regularity relation as $r_4(a,b,1,c)=0$. Hence, $(a+b-1-c)^2=4(ab+1)(c+1)$. Using \eqref{eq:xyz}, \eqref{eq:r}, and \eqref{eq:r-square}, we get 
\begin{equation}\label{eq:c-plus-one}
 c+1=(xy\pm r)^2.
\end{equation}

We now exclude the minus sign in \eqref{eq:c-plus-one}.  Since $r$ and $xy$ have the same degree and the same leading coefficient, it follows that
\begin{equation}\label{eq:xy+r}
 \deg(xy+r)=\deg x+
\deg y.
\end{equation}
Moreover, from \eqref{eq:r-square}, it follows
\begin{equation}\label{eq:xyr-factor}
 (xy-r)(xy+r)=x^2y^2-r^2=x^2+y^2-2.
\end{equation}
The polynomial on the right-hand side has degree $2\deg y$, since $\deg x\le\deg y$ and the leading coefficients of $x^2$ and $y^2$ cannot cancel over $\R$.  Combining this with \eqref{eq:xy+r} and \eqref{eq:xyr-factor}, we get
\[
 \deg(xy-r)=\deg y-
\deg x<\deg y,
\]
since $x$ is nonconstant.  But, $\deg z\ge\deg y$.  Hence, $z$ cannot be equal, even up to sign, to $xy-r$.  Replacing $z$ by $-z$ if necessary, we may assume that
\begin{equation}\label{eq:z=xyr}
 z=xy+r.
\end{equation}

Finally, we use the condition  that $abc+1$ is a square.  Since $ab=r^2-1$ and $c=z^2-1$, the polynomial
\begin{equation}\label{eq:rz-square}
 (r^2-1)(z^2-1)+1
\end{equation}
is a square in $\R[X]$.  The polynomials $r$ and $z=xy+r$ are nonconstant and have the same degree.  Applying Lemma \ref{lem:equal-degree}, with $p=r$ and $q=z$, we obtain two possible cases. Suppose first that
\[
 z=\pm r\pm1.
\] By \eqref{eq:z=xyr}, either $xy=\pm1$ or $xy+2r=\pm1$.  Both options are not possible, since $x$ and $y$ are nonconstant and  nonconstant polynomials $r$ and $xy$ have the same degree and the same leading coefficients. The other possibility is
\[z=\mu r,
\]
for some $\mu\in\{\pm\sqrt2\pm1\}.$ By \eqref{eq:z=xyr}, it easily follows $xy=(\mu-1)r.$ Hence, $\mu-1\ne0$.  Substituting $r=xy/(\mu-1)$ into \eqref{eq:r-square}, yields
\begin{equation}\label{eq:mu-comparison}
 \frac{x^2y^2}{(\mu-1)^2}=x^2y^2-x^2-y^2+2.
\end{equation}
Since $x$ and $y$ are both nonconstant, the degree of $x^2y^2$ is strictly larger than the degrees of $x^2$ and $y^2$.  Comparing the leading coefficients in \eqref{eq:mu-comparison}, gives $(\mu-1)^2=1$.  However, none of the four values $\mu\in\{\pm\sqrt2\pm1\}$  satisfies that condition.\end{proof}

\begin{corollary}
Let $A$ be a subring of $\R$ containing $1$.  Under the distinct nonzero convention, every exotic Diophantine triple in $A[X]$ with a nonconstant element contains $1$. In particular, this holds over $K[X]$, for every subfield $K\subseteq\R$. Since $2$ is not a square in $\mathbb{Q}[X]$, there are no exotic Diophantine triples in $\mathbb{Q}[X]$ with a nonconstant element.
\end{corollary}
\begin{proof}
Any exotic Diophantine triple in $A[X]$ is also one in $\R[X]$.  The conclusion follows directly from Theorem \ref{thm:main}.
\end{proof}

\begin{remark}
If the zero polynomial is allowed, then for any nonconstant $t\in\R[X]$ the set
\[
 \{0,\ t^2-1,\ (t+1)^2-1\}
\]
satisfies the seven-square condition \eqref{eq:seven} and contains no element equal to $1$.  It is not an exotic Diophantine triple from Definition \ref{def:exotic}, because exotic triples are required to have nonzero elements.
\end{remark}

\section{Acknowledgements}
This work was supported by the Croatian Science Foundation Grant No.~IP-2022-10-5008. A.~J. was also supported by the European Union -- NextGenerationEU, project number uniri-iz-25-62-ALGEBRA. During the preparation of this work, the author utilized ChatGPT-5.5 Pro to assist with searching for possible exotic triples. The author independently verified, refined, and wrote all mathematical arguments, and takes full responsibility for the correctness and originality of the final results. The author would like to thank Professor Andrej Dujella for his support and suggestions.

\end{document}